\documentclass[12pt,oneside]{amsart}
\usepackage{amscd,amsmath,amssymb,amsfonts,a4}
\usepackage[cmtip, all]{xy}
\usepackage{pstricks}
\usepackage{color}

\newlength{\rulebreite}

\def\timesover#1#2#3{\ \xymatrix@1@=0pt@M=0pt{ _{#1}&\times&_{#2} \\& ^{#3}&}\ }
\def\otimesover#1#2#3{\ \xymatrix@1@=0pt@M=0pt{ _{#1}&\otimes&_{#2} \\& ^{#3}&}\ }

\usepackage[all]{xy}
\theoremstyle{plain}
\newtheorem{thm}{Theorem}
\newtheorem{lem}[thm]{Lemma}
\newtheorem{cor}[thm]{Corollary}
\newtheorem{prop}[thm]{Proposition}
\theoremstyle{definition}

\newtheorem{rmk}[thm]{Remark}

\numberwithin{thm}{section}
\numberwithin{equation}{section}

\newcommand{\Spec}{{\rm Spec \,}}

\newcommand{\C}{{\mathbb C}}

\newcommand{\F}{{\mathbb F}}

\newcommand{\Q}{{\mathbb Q}}
\newcommand{\R}{{\mathbb R}}

\newcommand{\Z}{{\mathbb Z}}

\def\tilde{\widetilde}
\begin{document}

\title[Algebraic entropy]{Algebraic versus topological entropy for
surfaces over 
finite fields}
\author{H\'el\`ene Esnault}
\address{
Universit\"at Duisburg-Essen, Mathematik, 45117 Essen, Germany}
\email{esnault@uni-due.de}
\author{Vasudevan Srinivas}
\address{School of Mathematics, Tata Institute of Fundamental Research, Homi
Bhabha Road, Colaba, Mumbai-400005, India}
\email{srinivas@math.tifr.res.in}
\date{ May 14, 2011}
\thanks{The first author is supported by  the SFB/TR45 and
the ERC Advanced Grant 226257, the second author is supported by a J.C. Bose
Fellowship of the D.S.T.}
\begin{abstract}  We show that, as in de Rham  cohomology over the complex
numbers, the value of the
entropy of an automorphism of the surface over a finite field $\F_q $ is taken
on
the span of the N\'eron-Severi group inside of $\ell$-adic cohomology.

\end{abstract}
\maketitle
\section{Introduction}
If $X$ is a smooth projective surface over the field of complex numbers,
and $\varphi: X\to X$ is an automorphism,  then a notion of {\it entropy} of
$\varphi$ has been defined on the underlying topological manifold $X(\C)$
and shown to be the same as the following
cohomological definition (\cite{Gromov}, \cite{Yomdin}, see also \cite{Fried}
and \cite[Theorem~2.1]{DS}): let
$H^{2\bullet}(X(\C))=H^0(X(\C))\oplus
H^2(X(\C))\oplus H^4(X(\C))$ be the even degree de Rham cohomology. The
automorphism $\varphi$ acts linearly on $H^{2\bullet}(X(\C))$ and as the
identity on $H^0(X(\C)) \oplus H^4(X(\C))$. Thus the maximum absolute value of
the eigenvalues of  $\varphi$  is $\ge 1$. One defines the entropy $h(\varphi)$  to
be the maximum of the natural logarithm of those absolute values. It is then
$\ge 0$ and of interest are the cases when it is $>0$. Clearly, it can only
happen when $\varphi$ is not of finite order on $H^{2\bullet}(X(\C))$, so a
fortiori when $\varphi$ does not have finite order as an  automorphism of $X$.

Keiji Oguiso observed that Hodge theory implies that, {\it this maximum
is taken on the span of the N\'eron-Severi group inside of de Rham cohomology},
in fact on the transcendental part of de Rham cohomology, $\varphi$ has finite
order (see Proposition \ref{hodge} for a slightly more precise
statement). On the other hand, the definition of the entropy stated above is
clearly algebraic. One can replace de Rham cohomology by $\ell$-adic \'etale
cohomology in the definition. Taking then a ring $R\subset \C$ of finite type
over $\Z$ over which $(X, \varphi)$ has a model $(X_R, \varphi_R)$ such that
$X_R$ has good reduction at all closed points $s\in \Spec R$,   one sees that
the value of the entropy of $\varphi_s=\varphi_R\otimes_R \kappa(s) $ on
$H^{2\bullet}(X_{\bar s}, \Q_\ell)$ is taken on the $\Q_\ell$-span of the
N\'eron-Severi group inside of $\ell$-adic cohomology, where $X_s=X_R\times_R s,
X_{\bar s}=X_s\otimes_{\kappa(s)} \bar \kappa(s)$. 

We ask whether this property comes from the fact that over the finite field
$\kappa(s)$, $(X_s, \varphi_s)$ is the reduction mod $p$ of $(X, \varphi)$ or
whether it is true in general as a property of eigenvalues of  automorphisms 
acting on $\ell$-adic cohomology over finite fields. 
 
Our main result says:

\begin{thm}\label{thm1} Let $(X,\Theta)$ be a smooth, projective
polarized surface over a  finite field ${\F}_q$, $\varphi\in {\rm Aut}\,(X)$ an 
automorphism of the underlying surface, not necessarily preserving the polarization. Let 
 $\bar X=X\otimes_{\F_q}\bar{\F}_p$ be the corresponding 
surface over the algebraic closure $\bar {\F}_p$ of $\F_q$ (where 
$p$ is the characteristic). 

Let $\ell\neq p$ be a prime, and let 
\[V=V(X,[\Theta],\varphi)\subset  
[\Theta]^{\perp}\subset H^2_{{\rm \acute{e}t}}(\bar{X},\Q_{\ell}(1))\]
be the largest $\varphi$-stable subspace which is contained in the 
orthogonal complement of $[\Theta]\in 
H^2_{{\rm \acute{e}t}}(\bar{X},\Q_{\ell}(1))$ with 
respect to the cup product pairing 
\[H^2_{{\rm \acute{e}t}}(\bar{X},\Q_{\ell}(1))\otimes_{\Q_{\ell}}
H^2_{{\rm \acute{e}t}}(\bar{X},\Q_{\ell}(1))\to 
H^4_{{\rm \acute{e}t}}(\bar{X},\Q_{\ell}(2))\cong \Q_{\ell}.\]  

Then  $\varphi$ has finite order on  $V$.
\end{thm}
 As this formulation does not involve directly the $\Q_\ell$-span of the
N\'eron-Severi, which is not always liftable  to characteristic 0 even if
$X$,
defined over the finite field, is  so liftable, one sees that one can
reverse the classical argument sketched above to get the following corollary: 
\begin{cor}\label{cor}
Let $(Y,\Theta)$ be a polarized surface over an algebraically closed 
field $k$, and let $\varphi:Y\to Y$ be an algebraic automorphism of $Y$
(with $\varphi$ not necessarily preserving the polarization). 

Let $\ell$ be a prime, invertible in $k$, and let 
\[V=V(Y,[\Theta],\varphi)\subset  
[\Theta]^{\perp}\subset H^2_{{\rm \acute{e}t}}(\bar{X},\Q_{\ell}(1))\]
be the largest $\varphi$-stable subspace which is contained in the 
orthogonal complement of $[\Theta]\in H^2_{\acute{e}t}(Y,\Q_{\ell}(1))$ with 
respect to the cup product pairing 
\[H^2_{{\rm \acute{e}t}}(Y,\Q_{\ell}(1))\otimes_{\Q_{\ell}}
H^2_{{\rm \acute{e}t}}(Y,\Q_{\ell}(1))\to 
H^4_{{\rm \acute{e}t}}(Y,\Q_{\ell}(2))\cong \Q_{\ell}.\]  
Then $\varphi$ has finite order on  $V$.

\end{cor}

While the Hodge theoretic argument is purely abstract  (i.e., depends only
on the 
properties of Hodge structures, and not on geometric arguments), the arguments
we present 
in this note for proving Theorem \ref{thm1} rely  on the classification of smooth projective 
surfaces, on the fact that surfaces of general type  (over a finite field)
have a finite
group of automorphisms, on the Tate conjecture  for abelian surfaces, and,
unfortunately, 
on one argument involving lifting $K3$ surfaces to characteristic $0$. So, due to this one $K3$ 
case,  we can't say that we have a purely arithmetic proof of  Corollary \ref{cor} over $\C$. On 
the other hand, Theorem \ref{thm1} should follow from the standard conjectures
(see section \ref{ss:standard}).   
So, aside from its interest for entropy questions, it can
also be viewed as a motivic statement. To reinforce this viewpoint, we show in
section Theorem~\ref{even}
 \begin{thm} \label{thm:total}
 In the situation the Theorem~\ref{thm1}, the maximum of the absolute values of
the eigenvalues
of $\varphi$ on $\oplus_{i=0}^4 H^i_{{\rm \acute{e}t}}(\bar X, \Q_\ell)$ 
(with
respect to any complex embedding of $\Q_{\ell}$ is achieved on   the
$\Q_\ell$-span
of $\langle \varphi^n[\Theta], \ n\in \Z \rangle$, in 
$H^2_{{\rm \acute{e}t} } (\bar X, \Q_\ell)$. 

\end{thm}
We deduce of course the same theorem over any field:
\begin{cor}
 In the situation of Corollary~\ref{cor}, the maximum of the absolute values of
the eigenvalues
of $\varphi$ on $\oplus_{i=0}^4 H^i_{{\rm \acute{e}t}}(\bar Y, \Q_\ell)$ 
(with
respect to any complex embedding of $\Q_{\ell}$) is achieved on   the
$\Q_\ell$-span
of $\langle \varphi^n[\Theta], \ n\in \Z \rangle$, in 
$H^2_{{\rm \acute{e}t} } (\bar Y, \Q_\ell)$. 

\end{cor}

\noindent{\it Acknowledgements:} This work has been greatly influenced by the 
observation of Keiji Oguiso, and his willingness to explain to us a bit of entropy theory.
We benefited from discussions with Curt McMullen on entropy,  with Pierre
Deligne on the relation between Theorem \ref{thm1} and the standard
conjectures (see section~\ref{ss:standard}),  with M. S. Raghunathan on
algebraic groups. The first author thanks the 
Tata Institute for Mathematics 
and the second author  the Essen Seminar for Algebraic Geometry and
Arithmetics for hospitality.

\section{Some preliminaries and general reduction steps to prove Theorem
\ref{thm1}}
  As was already done in the formulation of the Theorem and its  Corollary, we
write $\varphi$   for the contravariant action of $\varphi$ on
cohomology; 
it should be clear from the context if $\varphi$ denotes  the automorphism, or the linear 
automorphism obtained from it in a specific 
linear representation. 

As in Theorem \ref{thm1}, one considers the action of $\varphi$ on $H^i(\bar X,
\Q_\ell)$, we may as well replace $\F_q$ by a finite extension, and thus we
will always assume that the N\'eron-Severi group $NS(\bar{X})$ 
is defined over $\F_q$.

We may also replace $\varphi$ by any power $\varphi^n$, $n\neq 
0$, without loss of generality. In particular, as already observed, the result
has content only when $\varphi$ acts on $H^2_{{\rm \acute{e}t}}(\bar{X},{\mathbb 
Q}_{\ell}(1))$ through a linear automorphism of infinite 
(multiplicative) order.

\begin{lem}\label{reduction} Suppose $[\Theta]\in NS(\bar{X})$ 
has a finite orbit under $\varphi$. Then $\varphi$ itself has finite 
order as an automorphism, so a fortiori it has finite order on the whole
cohomology  $H^*_{{\rm \acute{e}t}}(\bar{X},\Q_{\ell}(1))$, and Theorem
\ref{thm1} is trivially satisfied. 
\end{lem}
\begin{proof} Replacing $\varphi$ by a power, we may assume 
that 
\begin{enumerate} 
\item[(i)] the algebraic equivalence class $[\Theta]\in 
NS(\bar{X})$ is fixed by $\varphi$
\item[(ii)] there is a very ample line bundle ${\mathcal L}$ on 
$X$, which satisfies $\varphi^*{\mathcal L}\cong {\mathcal L}$, whose 
class in $NS(\bar{X})$ is $m[\Theta]$ for some positive integer $m$. 
\end{enumerate}
Here, (i) is clear. For (ii), first choose a very ample $\mathcal L$ on 
$X$ with class $m[\Theta]$ for some positive integer $m$. Now the orbit 
of $[{\mathcal L}]$ in ${\rm Pic}\,(\bar{X})$ under the group of 
automorphisms generated by $\varphi$ is contained in a fixed coset of 
${\rm Pic}^{\tau}(X)(\bar{\F}_p)$, consisting of $\F_q$-rational points 
in ${\rm Pic}(\bar{X})$, and this is a finite set. Now replacing 
$\varphi$ by a suitable positive power, we may assume that the class of 
$\mathcal L$ in ${\rm Pic}(\bar{X})$ is fixed. Then (ii) holds, 
since $\varphi^*{\mathcal L}$ and $\mathcal L$ are line 
bundles on $X$, which become isomorphic on $\bar{X}$, so that they 
are isomorphic on $X$.

In particular, from (ii), the automorphism $\varphi$ of $X$ yields a 
graded automorphism of the ring $A=\oplus_{n\geq 0}H^0(X,{\mathcal 
L}^{\otimes n})$. Conversely, $\varphi$ is the induced automorphism on 
$X={\rm Proj}\, A$, obtained from the graded ring automorphism of $A$. 
Now Lemma~\ref{easy} below finishes the argument.
\end{proof}

\begin{lem}\label{easy} Let $k$ be a finite field, $A=\oplus_{n\geq 
0}A_n$ a finitely generated graded algebra over $A_0=k$. Then any graded 
automorphism of $A$ has finite order.
\end{lem}
\begin{proof} Since $A$ is finitely generated over $k$, it is generated 
by $W=\oplus_{i=0}^nA_i$ for some $n$, where we note that $W$ is a finite vector
space. 
Any graded automorphism of $A$ restricts to a $k$-linear automorphism 
of $W$, and this restriction uniquely determines the graded automorphism. 
Thus we may identify the group of graded automorphisms with a subgroup of 
the finite group $GL(W)$.
\end{proof}

\begin{prop}\label{eigen} Let the notation be as in 
Theorem~\ref{thm1}. Let $i\geq 0$, $j\in \mathbb Z$.
\begin{itemize}
\item[(i)] The eigenvalues of $\varphi$ acting 
on any \'{e}tale cohomology group 
$H^i_{{\rm \acute{e}t}}(\bar{X},\Q_{\ell}(j))$ 
are algebraic integers, which are units (that is, invertible elements 
in the ring of algebraic integers). 
\item[(ii)] The characteristic polynomial of $\varphi$ on 
$H^i_{{\rm \acute{e}t}}(\bar{X},\Q_{\ell}(j))$ has integer coefficients, 
and is a monic polynomial with constant term $\pm 1$. This polyomial is 
independent of $j$, and of the chosen prime $\ell\neq p$. 
\item[(iii)] A similar conclusion holds for the characteristic 
polynomial of $\varphi$ acting on any $\varphi$-stable subspace of 
$H^i_{{\rm \acute{e}t}}(\bar{X},\Q_{\ell}(j))$, which can be defined 
using a projector in the ring of self-correspondences of $X$ (i.e., 
corresponds to a direct summand of the Chow motive of $X$ which, on base 
change to $\bar{\F}_p$, has only one non-zero \'{e}tale cohomology 
group).  In particular, this holds for the 
characteristic polynomial of $\varphi$ on $V$. 
\end{itemize}
\end{prop}
\begin{proof}
This is a standard consequence of the  Grothendieck-Lefschetz 
fixed point formula, applied to powers of $\varphi$, combined with 
Deligne's theorem on the eigenvalues of the  geometric Frobenius 
(``Weil's Riemann Hypothesis'').  In (ii) and (iii), the point is 
that the characteristic polynomial in question has rational 
coefficients, since it is determined by the action on cohomology of a 
Chow Motive, while its roots are algebraic integers (this argument 
appears in a paper of Katz and Messing \cite{Katz-Messing}). Over finite 
fields,  Deligne's theorem implies that the individual \'{e}tale 
cohomology groups do correspond to Chow motives (see 
\cite[ Theorem~2.1]{Katz-Messing},   and the proof in 
\cite[ Theorem~2.1]{Katz-Messing}  applies equally well to any 
Chow motive which is a summand of the motive of $X$, whose \'{e}tale 
cohomology is concentrated in one degree.
\end{proof}

Next, we consider the subspace of $V$ spanned by algebraic cycles.
\begin{prop} Assume we are in the situation of Theorem~\ref{thm1}.
Let 
\[V_{{\rm alg}}=\left(NS(\bar{X})\otimes\Q_{\ell}\right)\cap V.\]
 Then $V_{{\rm alg}}$ is stable under $\varphi$, and 
$\varphi$ on $V_{{\rm alg}}$ has finite order.
\end{prop}
\begin{proof} The $\Q_\ell$-vector space $V_{{\rm alg}}$ has a
 natural $\Z$-structure $NV$ defined by the maximal $\varphi$-stable
subgroup
\[NV\subset [\Theta]^{\perp}\subset NS(\bar{X}),\]
where $\perp$ is the orthogonal complement with respect to the 
intersection poduct on $NS(\bar{X})$. One has 
$V_{{\rm alg}}=NV\otimes\Q_{\ell}$, and this identification is  
 $\varphi$-equivariant.  

Now $NV$ comes equipped with the intersection product $NV\otimes NV\to 
\Z$, which is non-degenerate after $\otimes{\mathbb R}$, and is {\em 
negative definite}, by the Hodge index Theorem for divisors. This pairing is 
clearly $\varphi$-stable as well, so that $\varphi$ can be considered 
as an orthogonal transformation for a Euclidean space structure on 
$NV\otimes{\mathbb R}$. In particular it is semi-simple. Moreover 
all eigenvalues of $\varphi$ on  $NV\otimes\C$  are of absolute value 1. Since these eigenvalues are 
algebraic integers (in fact units), and the characteristic polynomial 
of $\varphi$ on $NV\otimes\Q$ has rational coefficients, the 
eigenvalues are in fact algebraic integers, all of whose conjugates have 
absolute value 1; thus they are roots of unity, by a 
well-known theorem of Kronecker. This finishes the proof. 
\end{proof}
One obtains the immediate corollary:
\begin{cor} \label{alg}
Theorem~\ref{thm1} holds whenever 
$$V=V_{{\rm alg}},$$ or equivalently, whenever 
$$NS(\bar{X})\otimes\Q_{\ell}=H^2_{{\rm \acute{e}t}}(\bar{X},\Q_{\ell}
(1)).$$
More generally, let 
$$V_{{\rm tr}}=V\cap NS(\bar{X})^{\perp},$$
where the orthogonal $\perp$ is taken in $H^2(\bar X, \Q_\ell)(1)$. 
This is a $\varphi$-stable subspace of $V$, which may be defined as the 
cohomology of a suitable Chow motive, and the conclusion of  Theorem~\ref{thm1}
is equivalent to a similar statement about the eigenvalues of $\varphi$ 
on the subspace $V_{{\rm tr}}$. 
\end{cor} 
 \begin{rmk} \label{rmk} The cup product also induces a 
non-degenerate symmetric bilinear form on $V_{{\rm tr}}$ with values in 
$\Q_{\ell}$. Since $\varphi$ is an automorphism of $X$, $\varphi$ acts 
as an orthogonal transformation of $V_{{\rm tr}}$ with respect to this bilinear 
form.  \end{rmk}

\section{Using classification of surfaces}

Now we consider the possibilites for the surface $X$, from the 
perspective of the Enriques-Bombieri-Mumford classification of surfaces 
in arbitrary characteristic (a convenient reference for most of what we 
need is the book \cite{Badescu}). Since we need only 
consider surfaces where  $V\neq V_{{\rm alg}}$, {\em we may assume that the Kodaira 
dimension of $\bar{X}$ is $\geq 0$}. 

As a consequence, from \cite[Corollary~10.22 ]{Badescu},  since 
 $X$ has Kodaira dimension $\geq 0$, the birational 
equivalence class of $\bar{X}$ has a unique non-singular minimal 
model, say $X_0$. Increasing the finite field $\F_q$, we may assume the 
model $X_0$, and the morphism $\bar{X}\to X_0$, are defined over 
$\F_q$. Since $\varphi$ acts on $X$, it acts on its function field, and
thus on this unique minimal model (\cite[ Theorem~10.21]{Badescu}). 
So the automorphism  $\varphi$ of
$\bar{X}$ descends to 
$X_0$, and the spaces $V_{{\rm tr}}$ of $\bar{X}$ and $X_0$ are naturally 
identified.  Hence {\em we are reduced to the case when $\bar{X}$ is 
itself minimal}, i.e., $\bar{X}$ does not contain any exceptional curves of 
the first kind.  

We may also assume {\em $X$ is not of general type}. Indeed, 
$\varphi$ yields a graded automorphism of the canonical 
model ${\rm Proj} \big(\oplus H^0(X,\omega_X^{\otimes n})\big)$ of $X$, which
(by 
Lemma~\ref{easy}) has finite order. Hence, some power of $\varphi$ is 
an automorphism which acts trivially on the function field of $X$, since 
it is trivial on the canonical model, which (if $\bar{X}$ is of 
general type) is birational to $X$. Thus a power of $\varphi$ agrees 
with the identity on a Zariski dense subset, and hence equals 
the identity.

Thus, we need only focus on the cases when the Kodaira dimension of 
$X$ is 0 or 1. 

\begin{prop}\label{kod=1}
In the situation of Theorem~\ref{thm1}, suppose $X$ has 
Kodaira dimension 1. Then the conclusion of Theorem~\ref{thm1} holds. 
\begin{proof}
Let $C={\rm Proj}\big(\oplus_{n\geq 0}H^0(X,\omega_X^{\otimes 
n})\big)$ (this graded ring is finitely generated; 
see \cite[ Theorem~9.9 ]{Badescu}). Then from classification (same result in 
\cite{Badescu}) there is a morphism 
\[f:X\to C\]
 which gives rise to an elliptic or quasi-elliptic 
fibration, i.e., the generic fiber is a regular projective curve 
which has arithmetic genus 1, and the geometric generic fiber is either 
an elliptic curve, or is an irreducible rational curve with an ordinary 
cusp (this can occur only in characteristics 2 and 3,  
 \cite[Theorem~ 7.18]{Badescu}). 

In the quasi-elliptic case, the Leray spectral sequence for \'{e}tale 
cohomology implies that 
\[NS(\bar{X})\otimes\Q_{\ell}=
H^2_{{\rm \acute{e}t}}(\bar{X},\Q_{\ell}(1)),\] 
since the stalks of $R^1f_*{\mathbb Q}_{\ell}(1)$ at geometric points of 
$C$ vanish.

Hence we may assume without loss of generality that {\em $f$ is an 
elliptic fibration}.

Now we note that $\varphi$ induces a graded automorphism of the 
canonical ring $\oplus_{n\geq 0}H^0(X,\omega_X^{\otimes n})$, and thus 
an automorphism of $C$, which we may also denote by $\varphi$, such 
that $f:X\to C$ is $\varphi$-equivariant. 

As usual, after replacing $\varphi$ by a power, we may assume (from 
Lemma~\ref{easy}) that the induced automorphism of the canonical ring 
(and thus of the base curve $C$) is trivial.

Now for any morphism $D\to C$, $\varphi$ acts in a canonical way on the 
total space of the base changed morphism $X\times_CD\to D$, preserving 
the fibers; we denote this induced automorphism also by $\varphi$.  
Hence, if $D\to C$ is a finite morphism of nonsingular curves, so that 
$X\times_CD$ is an integral projective surface, $\varphi$ also acts on 
the normalization of $X\times_CD$, which is a normal projective 
surface, denoted by $X_D$. Making a suitable such base 
change $g:D\to C$, and normalizing, we may arrange that the resulting 
elliptic fibration $X_D\to D$ has a section. Clearly the singular locus 
of the normal surface $X_D$ is stabilized (as a set) by the 
automorphism. Consider the minimal resolution of singularities 
$\tilde{X}\to X_D$. If we write it as the blow-up of some ideal sheaf 
whose radical defines the 
singular locus, we may assume (after replacing $\varphi$ by a power) 
that this ideal sheaf is stabilized by $\varphi$, so that $\varphi$ 
lifts canonically to an automorphism of the blow-up $\tilde{X}$ (since 
$\varphi$ clearly determines an automorphism of the Rees algebra 
sheaf).

The morphism $\tilde{X}\to X$ is a generically finite proper morphism 
between smooth projective surfaces, which is $\varphi$-equivariant. 
We may choose a polarization $[\tilde{\Theta}]$ for $\tilde{X}$ which is 
the sum of the pullback of $[\Theta]$ and a divisor class with support 
in the exceptional divisor of $\tilde{X}\to X_D$. (This is a 
consequence of the negative definiteness of the intersection 
pairing on the exceptional curves, and the Nakai-Moishezon ampleness 
criterion.)  Then the resulting space 
$V(\tilde{X},[\tilde{\Theta}],\varphi)$ contains  
$V=V(X,[\Theta],\varphi)$ as a $\varphi$-stable subspace.

Thus, we are further reduced to considering the situation where the map 
$f:X\to C$, determined by the canonical divisor of $X$, is an elliptic 
fibration which has a section, and $\varphi$ is an automorphism of $X$ 
preserving the fibers. 

Let $U\subset C$ be the maximal open subset over which $f$ is smooth, so 
that $f_U:f^{-1}(U)\to U$ is an abelian scheme of relative dimension 1. 
Let $\bar{C}$, $\bar{U}$, ${f^{-1}(\bar U)}$ be the 
corresponding schemes over $\bar{\F}_p$.   
 The localisation sequence 
\[\oplus_{s\in \Sigma} H^2_{X_{\bar s}}(\bar{X}, \Q_\ell)\to
H^2_{{\rm \acute{e}t}}(\bar{X},\Q_{\ell}(1))\to 
H^2_{{\rm\acute{e}t}}({f^{-1}(\bar U)},\Q_{\ell}(1))\]
is exact and $\varphi$-equivariant, where $\Sigma$ is the discriminant of $f$.
On one hand, each summand $H^2_{X_{\bar s}}(\bar{X}, \Q_\ell)$ is (up to a Tate
twist) 
the free abelian group on the irreducible components of the geometric fiber
$X_{\bar{s}}$, 
and $\varphi$ acts via a permutation on the classes of these components. Thus
$\varphi$ has finite order on $\oplus_{s\in \Sigma} H^2_{X_{\bar s}}(\bar{X},
\Q_\ell)$. 

On the other hand, 
any automorphism of (the total space of) the abelian scheme $f^{-1}(U)$, 
which is compatible with the structure morphism $f_U$, is 
the composition of a  group-scheme automorphism 
(which has finite order) and a translation, since this is true on the 
elliptic curve which forms the geometric generic fiber. Replacing 
$\varphi$ by a power, we may assume further that $\varphi$ acts on 
$f^{-1}(U)$ as a translation by a section, with respect to the abelian 
scheme structure. 

We claim that, in this situation, $\varphi$ is unipotent on
$H^2_{{\rm \acute{e}t}}(f^{-1}(\bar U),\Q_{\ell}(1))$.
Indeed, the Leray spectral sequence yields a 
$\varphi$-equivariant exact 
sequence
\[0\to H^1_{{\rm \acute{e}t}}(\bar{U},R^1f_*\Q_{\ell}(1))\to 
H^2_{{\rm \acute{e}t}}(f^{-1}(\bar U),\Q_{\ell}(1)) \to 
 H^0_{{\rm \acute{e}t}}(\bar{U},R^2f_*\Q_{\ell}(1))\to 0\]
The action of $\varphi$ on $H^0_{{\rm
\acute{e}t}}(\bar{U},R^2f_*\Q_{\ell}(1))$ is the identity, while the action of
$\varphi$ in $R^1(f_U)_*\Q_\ell$ is trivial.  

On the other hand, the composition
\[H^2_{{\rm \acute{e}t}}(\bar{X},{\mathbb Q}_{\ell}(1))\to 
H^2_{{\rm \acute{e}t}}(f^{-1}(\bar U),{\mathbb Q}_{\ell}(1))\to 
 H^0_{{\rm \acute{e}t}}(\bar{U},R^2f_*\Q_{\ell}(1))\cong {\mathbb Q}_{\ell}\]
is identified with the intersection product with the cohomology class of 
a geometric fiber, from the projection formula. In particular, $V_{{\rm tr}}$ 
is contained in the kernel of the restriction map
\[H^2_{{\rm \acute{e}t}}({f^{-1}(\bar U)},\Q_{\ell}(1)) \to 
 H^0_{{\rm \acute{e}t}}(\bar{U},R^2f_*\Q_{\ell}(1)),\]
that is, 
\[V_{{\rm tr}}\subset H^1_{{\rm \acute{e}t}}(\bar{U},R^1f_*\Q_{\ell}(1))\]
as a $\varphi$-stable subspace on which $\varphi$ acts trvially. So the
 the
eigenvalues of $\varphi$ on the whole group  $H^2_{{\rm \acute{e}t}}(\bar
X,\Q_{\ell}(1)$ are roots of unity and $\varphi$ has finite order on
$V_{{\rm tr}}$. This finishes the proof.

\end{proof}
\begin{rmk} Our proof shows that if $X$ is elliptic, the eigenvalues of
$\varphi$  are roots of unity on the whole 
$H^2_{{\rm \acute{e}t}}(\bar X, \Q_\ell)$, that is  $\varphi$ acts
quasi-unipotently on it. 
\end{rmk}
\end{prop}
\begin{prop}\label{k=0,other}
Suppose that, in the sitution of Theorem~\ref{thm1}, the 
surface $\bar{X}$ is minimal of Kodaira dimension 0, and 
$\bar{X}$ is not a K3 or abelian surface. Then Theorem~\ref{thm1} 
holds for $X$.
\end{prop}
\begin{proof} As stated in \cite[page~1]{Bombieri-Mumford}, the 
minimal surfaces with Kodaira dimension 0 fall into 4 classes: $K3 $
surfaces, Enriques surfaces (both of ``classical'' and ``non-classical'' 
type), abelian surfaces and surfaces fibered over their Albanese, which 
is an elliptic curve (and the fibrations are either elliptic or 
quasi-elliptic). 

In case the Albanese variety of $\bar{X}$ is an elliptic curve, 
we  may assume (after increasing $\F_q$ if needed) that $X$ has an 
$\F_q$-rational $\varphi$-fixed point. Then ${\rm Alb}\,(\bar{X})$ 
and the Albanese mapping are defined over $\F_q$, and $\varphi$ induces 
a unique (group-scheme) automorphism of the elliptic curve ${\rm 
Alb}\,(\bar{X})$ making  the Albanese mapping 
$\varphi$-equivariant. Since the automorphism group of an 
elliptic curve is finite, replacing $\varphi$ by a power, we reduce to 
the  situation  where the action on ${\rm Alb}(\bar X)$ is trivial.

Now we may argue just as in the proof of Proposition~\ref{kod=1}, using 
the Albanese mapping instead of the mapping deduced from the canonical 
ring. Again, the case when $V_{{\rm tr}}$ is possibly nontrivial is for an 
elliptic fibration, and a similar Leray spectral sequence argument goes 
through. 

In the case of Enriques surfaces, including the non-classical ones, in 
fact one has  $V=V_{{\rm alg}}$ (see \cite[Theorem~4]{Bombieri-Mumford}), 
by an argument of Artin involving the Brauer group, so we conclude by Corollary
\ref{alg}.
\end{proof}

\section{The case of an abelian surface}

Any automorphism of the abelian surface $\bar{X}$ is the 
composition of a group automorphism and a translation by a 
closed point, where the translation has finite order. Hence, 
increasing the finite field $\F_q$ and replacing 
$\varphi$ by a power, if necessary, we may assume $X$ is an 
abelian surface over $\F_q$, and $\varphi$ is a group-scheme 
automorphism of $X$. We may assume that all the endomorphisms of 
$X$ are defined over $\F_q$. Why is this possible?

We may also deduce (by again increasing $\F_q$ and replacing $\varphi$ 
by a power, if necessary) that the validity of Theorem~\ref{thm1} for 
$X$ depends only on the isogeny class of $\bar{X}$. This follows 
because for any $n>1$, $\varphi$ acts as an automoprhism of finite order 
on the $n$-torsion $X(\bar{\F}_p)[n]$, and thus for any 
isogeny $\bar{X}\to X'$, some power of $\varphi$ acts trivially on 
its kernel, and so  a power of $\varphi$ descends to a compatible
automorphism of 
$X'$. 

Let $F:\bar{X}\to\bar{X}$ be the {\em geometric Frobenius 
morphism} associated to $X$, considered as an $\F_q$-scheme; thus 
\[F:\bar{X}\to\bar{X}\] is an $\bar{\F}_p$-morphism of 
degree $q^2$, which acts on each $H^i_{{\rm \acute{e}t}}(\bar{X},\Q_{\ell})$, 
with (by Deligne's theorem) a characteristic polynomial with 
$\mathbb Z$-coefficients, whose  (algebraic integer) roots all have 
complex absolute value  $(\sqrt{q})^i$. We may assume without loss 
of generality that $q$ is an even power of $p$, so that these absolute 
values are integers.

  We  now define   $P(t) 
\in \Z[t] $ to be the (monic) {\it minimal polynomial} of $F$ viewed as an element of
the finite rank torsion-free  $\Z$-module ${\rm
End}(X)$. Thus $P(t)\in \Q[t]$ is the minimal polynomial of $F$ as an element in 
${\rm End}(X)\otimes \Q$, and $P(t) \in \Q_\ell[t]$ is the minimal polynomial of
$F$ as an element in 
\[{\rm End}(X) \otimes \Q_l\subset  {\rm End}(H^1(\bar X, \Q_\ell))={\rm End}(H^1(\bar
X, \Q_\ell(j)))\mbox{ for all $j\in \Z$.}\]
 From the  Tate's theorems \cite{Tate} (see in particular Theorem
2; see also \cite[ Appendix~1,~Theorem~3]{Mumford}), proving the Tate
Conjecture  for endomorphisms of abelian varieties over finite fields, we know in 
particular that {\em $P(t)$ has no multiple roots}, and is thus a 
product of distinct monic irreducible polynomials which are pairwise 
relatively prime. Equivalently, $F$ acts semisimply on 
$H^1_{{\rm \acute{e}t}}(\bar{X},\Q_{\ell})$.

We consider also the {\em characteristic polynomial} $\in \Z[t]$  of $F$,
as 
defined in \cite{Mumford}, \S19, Theorem~4, which is the same, viewed in $
\Q_\ell[t]$, as the characteristic  polynomial of
$F$ as an element in ${\rm End}(H^1(\bar{X}, \Q_\ell))$.

Since $\dim_{{\mathbb Q}_{\ell}} 
H^1_{{\rm \acute{e}t}}(\bar{X},\Q_{\ell})=4$, it has degree $4$.

 In case the minimal polynomial is irreducible over 
${\mathbb Q}$, the characteristic polynomial must be a power of this 
minimal polynomial $P(t)$, and so the degree of the minimal polynomial $P(t)$ must 
divide $4$.

We now distinguish between several cases.\\[.1cm]
{\sf Case 1:} 
 {\it The minimal polynomial $P(t)$ is reducible over $\mathbb 
Q$.} 

In this case, $\bar{X}$ is not simple, since its endomorphism 
algebra has zero divisors, and in fact $\bar{X}$ must then be 
isogenous to a product of two mutually non-isogeneous elliptic curves 
(this is the only way to have two mutually coprime factors of $P(t)$). 
But then $\bar{X}$ has a finite group of automorphisms as an 
abelian variety, since this is the case for an elliptic curve, and any 
automorphism of a product of two non-isogeneous elliptic curves is a 
product of automorphisms on each of the two factors. Hence, in this 
situation, $\varphi$ cannot have infinite order, and we have nothing to 
prove. \\[.1cm]
{\sf Case 2: } $P(t)$ {\it is a linear polynomial.}

Then the characteristic polynomial is a power of a linear polynomial, 
and from the Tate conjecture, this implies that $\bar{X}$ is 
isogenous to $E\times E$ for a supersingular elliptic curve $E$. But in 
this case, since the endomorphism algebra of $E$ is a quaternion 
division algebra, which has dimension 4 over $\mathbb Q$, the Picard 
number of $E\times E$ is 6, which is also the second Betti number; one 
thus has that $V=V_{{\rm alg}}$, and $V_{{\rm tr}}=0$,  so we conclude
with Corollary \ref{alg}. \\[.1cm]
{\sf Case 3:}  $P(t)$ {\it  is an irreducible quadratic polynomial over} 
$\mathbb Z$. 

Let $\lambda$ be a complex root of $P(t)$. Then $\lambda$ is a non-real 
complex number, with $|\lambda|^2=\lambda\bar{\lambda}=q$. Indeed, if 
$\lambda$ is a real root, it must be an integer, since $q$ is an even 
power of $p$; however $P(t)$ is irreducible. 

Hence we must have that $P(t)=(t-\lambda)(t-\bar{\lambda})$, where 
${\mathbb Q}(F)\cong {\mathbb Q}(\lambda)$ is an imaginary quadratic 
field. Clearly the characteristic polynomial of $F$ on 
$H^1_{{\rm \acute{e}t}}(\bar{X},\Q_{\ell})$ is just $P(t)^2$ (see 
\cite[  Appendix~1,~Theorem~3~(e)]{Mumford}).

Since $X$ is an abelian surface, the cup product gives an isomorphism
\[H^2_{{\rm \acute{e}t}}(\bar{X},\Q_{\ell})=\left(\bigwedge^2 
H^1_{{\rm \acute{e}t}}(\bar{X},\Q_{\ell})\right).\]
Thus we see that the characteristic polynomial of $F$ on 
$H^2_{{\rm \acute{e}t}}(\bar{X},\Q_{\ell})$ has to be 
\[(t-\lambda^2)(t-\bar{\lambda}^2)(t-|\lambda|^2)^4=
(t-\lambda^2)(t-\bar{\lambda}^2)(t-q)^4.\]

From the Tate conjecture for divisors on $\bar{X}$, we conclude 
that $V_{\rm tr}$ is 2-dimensional, and the characteristic polynomial of $F$ 
on $V_{\rm tr}$ is the quadratic polynomial
\[(t-\lambda^2)(t-\bar{\lambda}^2).\]
 
As noted before, the cup product on  
$H^2_{{\rm \acute{e}t}}(\bar{X},\Q_{\ell}(1))$ gives 
rise to a non-degenerate symmetric bilinear form on $V_{{\rm tr}}$ with 
values in ${\mathbb Q}_{\ell}$, and
$\varphi$ and $\frac{1}{q}F$ are orthogonal transformations with respect 
to this form, which commute.  

Now $\varphi$ is a unit in the ring of endomorphisms of the 
abelian variety $X$, and from the Tate conjecture, ${\rm 
End}(\bar{X})\otimes{\mathbb Q}$ is a central simple algebra of 
dimension $4$ over its centre $K={\mathbb Q}[F]$, the subalgebra 
generated by $F$ (this central simple algebra is either a matrix 
algebra of size 2,  or a quaternion divison algebra). Hence the reduced 
characteristic polynomial\footnote{This is the 
characteristic polynomial of $\varphi$, considered as an element of 
${\rm End}(\bar{X})\otimes_K\bar{K}$, which  
is the algebra of $2\times 2$ matrices over $\bar{K}$.} 
of $\varphi$, considered as an element of this endomorphism 
algebra, is a quadratic polynomial with coefficients in $K$,
\begin{equation}\label{eq1}
f(x)=x^2-{\rm Trd}(\varphi)x +{\rm Nrd}\,(\varphi)
\end{equation}
where ${\rm Trd}(\varphi)\in K$ and ${\rm Nrd}(\varphi)\in K$ are the values of
the reduced norm and trace of the 
central simple algebra.  

Now on 
\[H^1_{{\rm \acute{e}t}}(\bar{X},{\mathbb Q}_{\ell})\otimes{\bar {\mathbb 
Q}_{\ell}}\]
we may diagonalize $F$. Fixing an embedding of ${\mathbb 
Q}_{\ell}$ into the complex number field $\mathbb C$, we can
split the resulting complex vector space into its $\lambda$ and 
$\bar{\lambda}$ eigenspaces for the action of $F$. 

Clearly $\varphi$, considered as a complex linear transformation, 
stabilizes this decomposition, since  $\varphi$ commutes with $F$. 
Further, on each of the 2-dimensional $F$-eigenspaces, on which $F$ acts
as $\lambda\cdot {\rm Id}$ and $\bar \lambda \cdot {\rm Id}$,
$\varphi$ has the appropriate  characteristic polynomial (with 
$\mathbb C$-coefficients) $\sigma(f)$ or $\bar{\sigma}(f)$, where 
$\sigma$, $\bar{\sigma}$ are the embeddings of $K$ into $\mathbb C$ 
determined by $\sigma(F)=\lambda$, $\bar{\sigma}(F)=\bar{\lambda}$ 
(resulting in two conjugate embeddings $K[x]\hookrightarrow 
{\mathbb C}[x]$, denoted the same way). 

The upshot is that, on the 2-dimensional complex vector space 
\[V_{{\rm tr}}\otimes_{{\mathbb Q}_{\ell}}\mathbb C\] 
 $\varphi$ is diagonalizable, and has eigenvalues
$\sigma({\rm  Nrd}\,(\varphi))$ and $\bar{\sigma}({\rm Nrd}\,(\varphi))$.
But ${\rm Nrd}(\varphi)\in K$ is actually an algebraic integer, which is 
a unit. Since $K$ is an imaginary quadratic field, ${\rm Nrd}(\varphi)$ 
must be a root of unity. Thus $\varphi$ is semisimple on $V_{{\rm tr}}$, with 
eigenvalues which are roots of unity, and this finishes the proof of 
Theorem~\ref{thm1} in this case.\\[.1cm]
{\sf Case 4:}  $P(t)$ {\it is an irreducible polynomial over} $\mathbb Q$
{\it of 
degree 4.} 

In this case, $F$ has 4 distinct algebraic non-real eigenvalues on
$H^1_{{\rm \acute{e}t}}(\bar{X},{\mathbb Q}_{\ell})$, which 
(once we embed ${\mathbb Q}_{\ell}$ into $\mathbb C$) are of the form 
$\lambda$, $\bar{\lambda}$, $\mu$, $\bar{\mu}$, with 
$|\lambda|^2=|\mu|^2=q$. 

In this case, on $H^2_{{\rm \acute{e}t}}(\bar{X},{\mathbb Q}_{\ell})$, 
$F$ has the eigenvalues $\lambda\mu$, $\bar{\lambda}\mu$, 
$\lambda\bar{\mu}$, $\bar{\lambda}\bar{\mu}$, 
which are again all distinct and non-real, as well as the eigenvalue $q$ with
multiplicity 2. From the Tate conjecture for $\bar{X}$, we see that 
$V_{{\rm tr}}$ is a 4-dimensional space, on which $F$ acts with the above 4 
distinct non-real eigenvalues. 

Now the Tate conjecture implies (see \cite[Theorem~2]{Tate},  or 
\cite[Appendix~1,~Theorem~3 ]{Mumford})  that the minimal
and 
characteristic polynomials  of $F$ on $H^1_{{\rm \acute{e}t}}(\bar{X},{\mathbb
Q}_{\ell})$ coincide, and we have that ${\rm 
End}(\bar{X})\otimes{\mathbb Q}={\mathbb Q}(F)=K$. Hence for some
polynomial $f(t)\in {\mathbb Q}[t]$, 
we have that $\varphi=f(F)\in K$. 

Fixing an embedding of ${\mathbb Q}_{\ell}$ into $\mathbb C$, 
we may choose a basis of eigenvectors $\{v_{\lambda}, 
v_{\bar{\lambda}},v_{\mu}, v_{\bar{\mu}}\}$for $F$
on $H^1_{{\rm \acute{e}t}}(\bar{X},{\mathbb Q}_{\ell})\otimes\mathbb C$, 
indexed by the corresponding eigenvalues. Then these are 
also eigenvectors for $\varphi$, with eigenvalues 
$f(\lambda),f(\bar{\lambda}), f(\mu), f(\bar{\mu})$ respectively.

Now $V_{{\rm tr}}\otimes\mathbb C$ has a resulting basis
$\{v_{\lambda}\wedge v_{\mu},
v_{\lambda}\wedge v_{\bar{\mu}},
v_{\bar{\lambda}}\wedge v_{\mu},
v_{\bar \lambda}\wedge v_{\bar \mu} \}.$
For this basis, it is then clear that $\varphi$ acts diagonally, 
with eigenvalues $f(\lambda)f(\mu)$, $f(\lambda)f(\bar{\mu})$, etc. 

We now observe that {\em $K$ is a CM field}, i.e., a totally non-real 
quadratic extension of a totally real number subfield.  Indeed, the 
distinct 
embeddings of $K$ into $\mathbb C$ are determined by $F\mapsto 
\lambda$, $F\mapsto \mu$, and their complex conjugate embeddings, so $K$ 
is totally non-real. It is also clear that the subfield $L={\mathbb 
Q}(F+\frac{q}{F})$ is totally real, and $K$ is a quadratic extension, 
since $|\lambda|^2=|\mu|^2=q$.

Since $\varphi\in K$ is an automorphism of $X$, it is a unit in the ring 
of integers ${\mathcal O}_K$. From the {\em Dirichlet unit theorem} 
(see for example \cite[ Chapter~2,~Theorem 5 ]{BS}), the 
unit groups of ${\mathcal O}_K$ and of the integers ${\mathcal O}_L$ in the
totally real 
subfield $L$ have the same rank. This means that, after 
replacing $\varphi$ by some power, we may assume $\varphi$ lies in $L$, 
and all of its eigenvalues on $H^1_{{\rm \acute{e}t}}(\bar{X},{\mathbb 
Q}_{\ell})\otimes\mathbb C$ are {\em real algebraic numbers}.

Hence on $H^1_{{\rm \acute{e}t}}(\bar{X},{\mathbb Q}_{\ell})\otimes 
\mathbb C$, $\varphi$ has 
two distinct eigenvalues $f(\lambda)=f(\bar{\lambda})$ and 
$f(\mu)=f(\bar{\mu})$, each with multiplicity 2 (since $\varphi$ has 
infinite order, and determinant
$1$ (as the degree of $\varphi$ is $1$), these two real numbers must be 
distinct, and satisfy $f(\lambda)^2f(\mu)^2=1$). But this implies 
$\varphi$ acts on 
$V_{{\rm tr}}\otimes\bar{{\mathbb Q}_{\ell}}$ as the real scalar 
$f(\lambda)f(\mu)$, which must be $\pm1$.

\section{The case of a $K3$ surface}

In the proof of Theorem~\ref{thm1}, there is one case remaining: the 
case when $X$ is a $K3$ surface. As in \cite{Bombieri-Mumford}, 
this means $\bar{X}$ is a smooth, projective minimal surface, and 
we have the properties
\[\omega_X\cong {\mathcal O}_X,\;\; H^1(X,{\mathcal 
O}_X)=0, \;\;
\dim_{\Q_{\ell}}H^2_{{\rm \acute{e}t}}(\bar{X},\Q_{\ell}(1))=22,\;\; {\rm 
Pic}^{\tau}(\bar{X})=0.\]

 We first treat the case of a  supersingular $K3$ surface in the sense of
Shioda. Then by definition of supersingularity (in this sense)
$H^2(\bar X, \Q_\ell(1))$ is algebraic and we can apply
Corollary~\ref{alg}.

We now rely on the crutch of lifting to characteristic 0. From a recent paper
\cite{Lieblich-Maulik} (see in particular Theorem~6.1 and the bottom of
page 8), it follows that  if $X$ is not a Shioda-supersingular $K3$
surface,  we can find
\begin{enumerate}
\item[(i)] a complete discrete valuation ring $R$, with residue field
$\bar{\F}_p$, 
and quotient field of characteristic 0
\item[(ii)] an $R$-scheme $\pi:{\mathcal X}\to {\rm Spec}\,R$, such that $\pi$
is projective 
and smooth, of relative dimension $2$, with closed fiber  $\bar{X}$
\item[(iii)] if $Y:={\mathcal X}_{\bar{\eta}}$ is the geometric generic
fiber of $\pi$,
then the specialization homomorphism ${\rm Pic}\,(Y)\to {\rm
Pic}\,(\bar{X})$ is an isomorphism, which 
induces an isomorphism between the respective cones of effective cycles
\item[(iv)] there is an {\em injective} specialization homomorphism ${\rm
Aut}\,(Y)\to{\rm Aut}\,(\bar{X})$, 
whose image has finite index.
\end{enumerate}

The specialization map on automorphisms in (iv), which is important for us here,
is defined as follows. If $\psi$ is an 
automorphism of $Y$, then (after making a base change if needed), the authors of
\cite{Lieblich-Maulik} prove that it is induced by an 
automorphism of the generic fiber ${\mathcal X}_{\eta}$, which extends to an
$R$-automorphism $\psi_{\mathcal X}$ of ${\mathcal X}\setminus S$ for some 
finite set $S\subset \bar{X}\subset {\mathcal X}$ of closed points; the
induced automorphism of $\bar{X}\setminus 
S$ then extends to an automorphism of $\bar{X}$, which is defined to be the
specialization of $\psi$.

Granting this, we see that, after replacing $\varphi$ by a power, if necessary,
we may assume $\varphi$ is the 
specialization of an automorphism of $Y$, in the above sense. It then follows
that, under the specialization isomorphism 
\[H^2_{{\rm \acute{e}t}}(Y,{\mathbb Q}_{\ell}(1))\cong
H^2_{{\rm \acute{e}t}}(\bar{X},{\mathbb Q}_{\ell}(1))\]
the respective actions of $\psi$ and $\varphi$ are compatible. Further, the
polarization $[\Theta]$ of $\bar{X}$ 
determines uniquely a polarization of $Y$, which we may also denote by
$[\Theta]$, compatibly with the specialization 
isomorphism. The specialization isomorphism above is of course one component of
an isomorphism between cohomology rings, 
and so respects the corresponding cup products, thus inducing also an
isomorphism of $\ell$-adic vector spaces 
\[V(Y,[\Theta],\psi)\cong V(\bar{X},[\Theta],\varphi),\] 
again compatible with the respective automorphisms $\psi$, $\varphi$. It thus
suffices 
to prove that the eigenvalues of $\psi$ on $V(Y,[\Theta],\psi)$ are roots of
unity. 

We may identify the algebraic closure of the quotient field of the DVR $R$ with
the complex number field $\mathbb C$, and 
thus also consider $\psi$ as an automorphism of the complex projective $ K3$
surface $Y$. 

In fact, one has the following more general assertion; this observation is, in a
sense, the motivation for the Theorem 
proved in this paper, and was explained to us by K. Oguiso (in the shape
that $\varphi$ on $H^2_{{\rm tr}}(Y, \C)$ has finite order): 
\begin{prop}\label{hodge} Suppose $\psi$ is an automorphism of a projective
smooth surface $Y$ over $\mathbb C$, with a polarization 
$\Theta$ (not necessarily invariant under $\psi$). Then 
 $\psi$ has finite order on $V(Y,[\Theta],\psi)$.
\end{prop}
\begin{proof} 
By the comparison theorem between \'{e}tale and singular cohomology, we 
reduce to proving a similar assertion for the action of
$\psi$ on $H^2(Y,{\mathbb Q})$. In other words, it suffices to show that the
eigenvalues of $\psi$ acting on the
similarly defined ${\mathbb Q}$ vector space 
\[V(Y,[\Theta],\psi)\subset H^2(Y,{\mathbb Q})\]
are roots of unity. Since $\psi$ is also compatible with the cup product, it
defines an orthogonal transformation with 
respect to the non-degenerate bilinear form on $V$ defined by the cup product.

Consider now the associated non-degenerate real bilinear form on $V_{\mathbb
R}=V\otimes_{\mathbb Q}{\mathbb R}$. From the
Hodge decomposition, we may write $V_{\mathbb R}$ as an orthogonal direct sum
\[V_{\mathbb R}=V_{\mathbb R}^{(1,1)}\perp V_{{\mathbb R},{\rm tr}},\]
where 
\[V_{\mathbb R}^{(1,1)}=V_{\mathbb R}\cap H^{(1,1)} \subset H^2(Y,{\mathbb
C}),\]
and the other summand is its orthogonal complement. This does give an orthogonal
direct sum decomposition of $V_{\mathbb R}$,
since by the Hodge index theorem (the Hodge theoretic version), the cup product
pairing on $V_{\mathbb R}$ is negative definite
on $V_{\mathbb R}^{(1,1)}$, and positive definite on its orthogonal complement. 

Since the Hodge decomposition on $H^2(Y,{\mathbb C})$ is also preserved by
$\psi$, it follows that $\psi$ preserves the 
above orthogonal direct sum decomposition of $V_{\mathbb R}$. Hence, after
changing the sign of the inner product on 
$V_{\mathbb R}^{(1,1)}$, we see that $\psi$ preserves a non-degenerate Euclidean
form on $V_{\mathbb R}$. Hence the 
 $\psi$ is semi-simple and all its eigenvalues are complex numbers of
absolute value 1.

However, we also know that the eigenvalues of $\psi$ are algebraic integers,
which are invertible, and the characteristic 
polynomial of $\psi$ has integer coefficients (since it obviously has rational
coefficients). Thus, by Kronecker's theorem, 
these eigenvalues are roots of unity.
\end{proof}
\section{Some further remarks}
\subsection{Standard Conjectures and Theorem~\ref{thm1}} \label{ss:standard}
P. Deligne explained to us that our Theorem~\ref{thm1} would be  a consequence
of the standard conjectures, were they available. We reproduce his argument. 

As explained in Section $2$, we have to show that $\varphi$ has finite order on
transcendental cohomology  $H^2_{{\rm tr}}(\bar X, \Q_\ell(1))$, where $X$ is a
smooth projective surface over a finite field $\F_q$. We denote by $M$ the
underlying Chow motive with $\Q$ coefficients, which is endowed with a quadratic
form  $b: M\otimes M\to \Q$, which induces the cup-product $H^2_{{\rm tr}}(\bar
X, \Q_\ell(1))\otimes H^2_{{\rm tr}}(\bar X, \Q_\ell(1))\to H^4(\bar X,
\Q_\ell(2))$ in $\ell$-adic realization. The automorphism $\varphi$ induces an
orthogonal automorphism of $M$. Its characteristic polynomial lies
in $\Q[t]$. 
But the $\ell$-adic realization of the characterisitc polynomial lies in
$\Z[t]$ (\cite{Katz-Messing}),
thus in fact, it lies on $\Z[t]$.  On the other hand, there should
exist
(\cite[V~2.4.5.1(iv)]{Saavedra}) a fiber functor $\omega$ over $\R$ on the
category of Chow motives of weight $0$, with the extra property that $b(\omega)$
is a positive definite form. Thus this implies already that $\varphi$ is
semi-simple and that its eigenvalues on $M$ have absolute value $1$. On the
other hand, they are algebraic integers again by \cite{Katz-Messing}. 
 We conclude by Kronecker's theorem that the
eigenvalues are roots of unity.

\subsection{Topological entropy,  even and odd degree cohomology}
\label{ss:top}
Recall that the {\em topological entropy} of a homeomorphism $\varphi:M\to M$ of
a compact, orientable manifold $M$ is defined to be 
the natural logarithm of the spectral radius of the linear tranformation induced
by $\varphi$ on the rational cohomology 
algebra $H^\bullet (M,{\mathbb Q})$. Since $\varphi$ is a homeomorphism, it
induces a  ($\Z$-linear)
automorphism of the integral cohomology algebra, so 
that the characteristic polynomial of $\varphi$ acting on cohomology has integer
coefficients, and the eigenvalues of $\varphi$ on 
cohomology are algebraic integers which are invertible, that is, are units in
the ring of algebraic integers.

If $M$ is a complex smooth projective variety, and
$\varphi$ is an algebraic
automorphism, the value of the entropy is taken on the {\it even} degree
cohomology $H^{2\bullet}(M, \Q)$ (see \cite[Theorem~2.1]{DS}).

We can now go through our proof of Theorem~\ref{thm1} from which we deduce:
\begin{thm} \label{even} 
 In the situation the Theorem~\ref{thm1}, the maximum of the absolute values of
the eigenvalues
of $\varphi$ on $\oplus_{i=0}^4 H^i_{{\rm \acute{e}t}}(\bar X, \Q_\ell)$ with
respect to any
complex embedding is achieved on   the $\Q_\ell$-span
of $\langle \varphi^n[\Theta], \ n\in \Z \rangle$, in 
$H^2_{{\rm \acute{e}t} } (\bar X, \Q_\ell)$.

\end{thm}

\begin{proof}
 The automorphism $\varphi$ acts as the identity on 
  $H^i_{{\rm
\acute{e}t} }(\bar X, \Q_\ell)$, for $i=0$ and $i=4$. 
Since $\varphi$ respects the cup-product $H^1_{{\rm
\acute{e}t} }(\bar X, \Q_\ell) \times H^3_{{\rm
\acute{e}t} }(\bar X, \Q_\ell) \to H^4_{{\rm
\acute{e}t} }(\bar X, \Q_\ell)$, its eigenvalues on $H^3_{{\rm
\acute{e}t} }(\bar X, \Q_\ell) $ are the inverse of its eigenvalues on 
$H^1_{{\rm
\acute{e}t} }(\bar X, \Q_\ell)$.
 On the other hand, the
characteristic polynomial of $\varphi$ on any $H^i_{{\rm
\acute{e}t} }(\bar X, \Q_\ell)$ has $\Z$-coefficients, and the eigenvalues lie
in $\bar \Z$. Thus the constant term of this polynomial is $\pm 1$ and, fixing a
complex embedding of a number field containing all the roots, at least one
eigenvalue has absolute value $\ge 1$. Thus the maximum of the absolute values
is always achieved on $\oplus_{i=1}^3 H^i_{{\rm \acute{e}t}}(\bar X, \Q_\ell)$.

 By Theorem~\ref{thm1}, we just have to
see that the absolute values of the eigenvalues on $H^1_{{\rm \acute{e}t}}(\bar
X, \Q_\ell)$ and   $H^3_{{\rm \acute{e}t}}(\bar
X, \Q_\ell)$
are at most those on $H^2_{{\rm \acute{e}t}}(\bar X,
\Q_\ell)$.

Again we may assume after a finite extension of $\F_q$ that $X$ has a rational
fixed point under $\varphi$, which we take to define the Albanese mapping 
${\rm alb}: X\to {\rm Alb}(X)$. Then the action of $\varphi$ extends so as
to make ${\rm alb}$ a $\varphi$-equivariant map.

If the image of  $\rm alb$  is $0$, this means  $H^1_{{\rm
\acute{e}t}}(\bar X,
\Q_\ell)= H^3_{{\rm \acute{e}t}}(\bar X,
\Q_\ell)=0$, there is nothing to prove.

If the image of ${\rm alb}$ is a curve $C$, then $\varphi$ acts on $C$, thus on
its normalization $\tilde{C}$. Since the genus of $\tilde{C}$ is $\ge 1$, the
action of $\varphi$ on $\tilde{C}$, thus on $C$ has finite order. Thus via the
surjective  pull-back map ${\rm alb}^*: H^1_{{\rm \acute{e}t}}(\bar C,
\Q_\ell) \to H^1_{{\rm \acute{e}t}}(\bar X,
\Q_\ell)$,  and its injective push-down dual map 
${\rm alb}_*: H^3_{{\rm \acute{e}t}}(\bar X,
\Q_\ell) \to H^1_{{\rm \acute{e}t}}(\bar C,
\Q_\ell)$
the action of $\varphi$ on $
H^i_{{\rm \acute{e}t}}(\bar X,
\Q_\ell), \ i=1,3$ is finite as well.

 If this is $2$-dimensional, then either $X$ is
of general type, in which case $\varphi$ has finite order and there is nothing
to prove, or else $X$ is an abelian surface.  In this case, we have a more
general 
 Proposition~\ref{H^1-abelian} below on $H^1$. But for an abelian surface, 
$H^i_{{\rm \acute{e}t}}(\bar X,
\Q_\ell)=H^j_{{\rm \acute{e}t}}(\bar X^\vee,
\Q_\ell) $ for $(i,j)=(3,1) $ and $(2,2)$, where $X^\vee$ is the dual abelian
surface. Since the eigenvalues of the induced autmorphism $\varphi^\vee$ on 
$H^2_{{\rm \acute{e}t}}(\bar X^\vee,
\Q_\ell) $ are those of $\varphi$ on  $H^2_{{\rm \acute{e}t}}(\bar X,
\Q_\ell)$, Proposition~\ref{H^1-abelian} concludes the proof of
Theorem~\ref{even}.

\end{proof}
 We now show that for automorphisms of abelian varieties, the spectral
radius of the 
induced linear automorphism on $H^1$ is at most that for the similar linear automorphism
of $H^2$. 
\begin{prop}\label{H^1-abelian}
Let $X$ be an abelian variety over a field $k$, and $\varphi$ an automorphism of $X$. Let 
$\bar X=X\otimes_k \bar k$ be the corresponding variety over an
algebraic
 closure $\bar k$, and let $\ell$ be a prime invertible in $k$. 

Then the complex absolute values of the eigenvalues of $\varphi$ on $H^1_{{\rm
\acute{e}t}}(\bar X,\Q_{\ell})$ 
are bounded above by the maximum  of the complex absolute values of the eigenvalues of $\varphi$ on 
$H^2_{{\rm \acute{e}t}}(\bar X,\Q_{\ell})$.
\end{prop}
\begin{proof}
By the standard arguments involving the choice of a model over a finitely generated $\Z$-algebra, and specialization,
we reduce to the case when $k=\F_q$ is a finite field. We also fix an embedding $\Q_{\ell}\hookrightarrow \C$, so that
we may speak of the eigenvalues as complex numbers. Without loss of generality, we may also increase the size of the 
finite field $\F_q$, replace $\varphi$ by a power, and replace $X$ by an isogenous abelian variety. Thus, we may write 
$X=X_1\times\cdots\times X_r$ where the $X_i$ are powers of mutually
non-isogenous absolutely simple abelian varieties, 
in which case $\varphi$ must be a product $\varphi_1\times\cdots\times\varphi_r$ with $\varphi_j\in {\rm Aut}(X_j)$.
From the K\"unneth formula, it follows that it suffices to consider the case
when $X=X_1$ is a power of an absolutely simple
abelian variety.  In this case, ${\rm End}(X)\otimes\Q$ is a central simple algebra over a number field. 

We also make use of the fact that $H^2_{{\rm \acute{e}t}}(\bar
X,\Q_{\ell})=\bigwedge^2H^1_{{\rm \acute{e}t}}(\bar X,\Q_{\ell})$ for an abelian
variety. The automorphism $\varphi$ has eigenvalues 
on $H^1_{{\rm \acute{e}t}}(\bar X,\Q_{\ell})$  which are
invertible algebraic integers whose product is $1$, and so the maximal absolute
value 
of these eigenvalues is always $\geq 1$.

Thus, if we consider the complex absolute values of the eigenvalues of $\varphi$
on $H^1_{{\rm \acute{e}t}}(\bar X,\Q_{\ell})$, counted with multiplicity, the 
Proposition is clearly true, unless the largest such absolute value is $>1$, and appears exactly once, 
while all the other absolute values are $<1$. Since the set of eigenvalues is closed under complex 
conjugation (as the characteristic polynomial of $\varphi$ has integer coefficients), this largest
absolute value must correspond to a real eigenvalue, which we may take to be positive (replace $\varphi$ 
by its square if needed).

In other words, we have to rule out the possibility that $\varphi$ acting on
$H^1_{{\rm \acute{e}t}}(\bar X,\Q_{\ell})$ has one real eigenvalue $\lambda>1$, 
occuring with multiplicity 1, and all other eigenvalues of complex absolute value $<1$ (in particular, $\lambda$ must 
be a ``Pisot-Vijayaraghavan number''.) We do this by induction on the dimension
of $X$. 
Let $P(t)\in \Z[t]$ be the monic minimal polynomial of $\varphi$ as an element
of ${\rm End}(\bar X)$, and let 
$f(t)\in \Z[t]$ be the monic minimal polynomial over $\Q$ for the real algebraic integer $\lambda$. Then there is a
factorization of polynomials $P(t)=f(t)g(t)$, since $\lambda$ is an eigenvalue for $\varphi$, so that $P(\lambda)=0$. 
Now $\lambda$ must be a simple root of $P(t)$, so that $f(t)$, $g(t)$ are relatively prime polynomials in $\Q[t]$. If 
$g(t)$ is non-constant, then the identity component $Y$ of the
subgroup-scheme $\ker f(\varphi) \subset X$ is a  subabelian variety of
dimension $\ge 1$ and $< {\rm dim} X$  which is $\varphi$-stable and such that 
$\lambda >1$ is an eigenvalue of $\varphi$ on 
$H^1_{{\rm \acute{e}t}}(\bar Y,\Q_{\ell}) $. Thus $Y$ has dimension $\ge 2$ and
we can replace $X$ by $Y$ to show Proposition~\ref{H^1-abelian}, that is we can 
assume that $g(t)$ is constant, so that $P(t)=f(t)$ (as they are both monic).

In particular, the subring $L\subset {\rm End}(X)\otimes\Q$ generated by $\varphi$ over $\Q$, is a subfield, isomorphic to
$\Q(\lambda)$. We must also have 
\[[L:\Q]={\rm deg}P(t)=\dim_{\Q_{\ell}}H^1_{{\rm \acute{e}t} }(\bar X,\Q_{\ell})
=2\dim X.\] 
Thus $\varphi$ has distinct eigenvalues on $H^1_{{\rm \acute{e}t} }(\bar
X,\Q_{\ell})$, and is diagonalizable, and $\Q_{\ell}(\varphi)$ is a
maximal 
commutative subring of ${\rm End}_{\Q_{\ell}}(H^1_{{\rm \acute{e}t} }(\bar
X,\Q_{\ell}))$. In particular $L\subset {\rm End}(\bar X)$ is also a
maximal 
commutative subring. Thus 
     the
geometric Frobenius $F \in {\rm End}(\bar X)$, which commutes with $\varphi$, 
lies in $L$, and $F=Q(\varphi) $ for some polynomial $Q(t)\in \Q[t]$. We
conclude that $F$ has the eigenvalue $Q(\lambda) \in \R$ on 
$H^1_{{\rm \acute{e}t} }(\bar
X,\Q_{\ell})$. 
 This means, 
assuming, as we may, that $q$ is an even power of $p$, that $F$ has an integer
eigenvalue. Since $X$ is a power of an absolutely
simple abelian variety, Tate's theorems imply that the minimal polynomial of $F$
in ${\rm End}(\bar X)\otimes\Q$ is irreducible,
and  so, having an integer root, must be a linear polynomial. This forces $X$
to be isomorphic to a power of a supersingular  elliptic curve, say 
$X\cong E^n$. 

Now ${\rm End}(\bar X)\otimes\Q\cong M_n(D)$, where $D={\rm
End}(\bar E)\otimes\Q$ is the unique quaternion division algebra 
over $\Q$ which splits at all places apart from $p$ and $\infty$. Since
$L\subset M_n(D)$ is a maximal commutative subfield of the central 
simple algebra $M_n(D)$, we know that $L$ is a splitting field for the algebra, i.e., $M_n(D)\otimes_{\Q}L\cong  {\rm End}_L(M_n(D))\cong 
M_{2n}(L)$ as central simple algebras over $L$ (where $D$ is regarded as an $L$-vector space through right multiplication; the isomorphism
is given by $M_n(D)\otimes L\ni a\otimes b\mapsto (x\mapsto axb)\in {\rm End}_L(M_n(D))$). 
(We thank M. S. Raghunathan for a discussion on this point.).
Since $L$ has a real embedding, we conclude that $M_n(D)\otimes_{\Q}{\mathbb R}\cong M_{2n}({\mathbb R})$, which contradicts that $D$ is 
non-split at $\infty$.  This concludes the proof.
\end{proof}

\subsection{Algebraic entropy} \label{ss:alg}

In general, if $X$ is a smooth proper variety over a field $k$, and $\varphi$ an
algebraic automorphism of $X$, then we may associate 
to it two numerical invariants, as follows.
\begin{enumerate}
 \item Let $\ell$ be a prime invertible in $k$, and let
$\bar{X}=X\times_k\bar{k}$ be the corresponding (smooth, proper)
variety over 
an algebraic closure $\bar{k}$. The characteristic polynomial of $\varphi$
on $H^\bullet_{{\rm \acute{e}t}}(\bar{X},{\mathbb Q}_{\ell})$ is 
independent of $\ell$, and has integer coefficients, and algebraic integer roots
(which are units); hence we may define the spectral radius 
of $\varphi$ on $H^\bullet_{{\rm \acute{e}t}}(\bar{X},{\mathbb Q}_{\ell})$ as a
real number $\geq
1$, and define its natural logarithm to be the {\em topological entropy}
of $\varphi$. When $k\subset {\mathbb C}$, so that we may associate to
$(X,\varphi)$ a compact complex manifold $X_{\mathbb C}$, and a holomorphic 
automorphism $\varphi_{\mathbb C}$, then our definition agrees with the usual
one (given above) for $\varphi_{\mathbb C}$.
\item  We may instead define an invariant using algebraic cycles, as follows.
Let $\bar{X}$ be as above, and $CH^\bullet_{\rm num}(\bar{X})$ the ring of 
algebraic cycles on $\bar{X}$ modulo numerical equivalence. Then 
$\varphi$ yields an automorphism of the ring $CH^\bullet_{\rm num}(\bar{X})$,
whose underlying abelian group is known to be a finitely generated free abelian 
group; thus the characteristic polynomial of $\varphi$ on this ring has integer
coefficients, and eigenvalues which are algebraic integer units. 
We may now define the {\em algebraic entropy} of $\varphi$ to be the natural
logarithm  of the spectral radius of $\varphi$ acting on $CH^\bullet_{\rm
num}(\bar{X})$.
\end{enumerate}

Our main result, Theorem~\ref{thm1}, and its corollary ~\ref{cor},
with Theorem~\ref{even}, imply that for automorphisms of smooth projective 
algebraic surfaces, {\em the algebraic
and topological entropies coincide}. One may ask whether this is true
in arbitrary dimension.  
 It would in particular imply that the value of the entropy on the whole
$\ell$-adic cohomology is taken on even degree cohomology, which is true in
characterisitc $0$ (see Section~\ref{ss:top}).

\bibliographystyle{plain}
\renewcommand\refname{References}

\end{document}